\documentclass{jgcc}
\pdfoutput=1
 \usepackage{lastpage}
 \jgccdoi{13}{2}{3}{8796}  
 \jgccheading{}{\pageref{LastPage}}{}{}%
 {Mar.~16,~2021}{Nov.~19,~2021}{}
 

\usepackage{amsthm}
\usepackage{amsfonts}
\usepackage{amssymb}
\usepackage{amsmath}
\usepackage{hyperref}
\usepackage{graphicx}
\usepackage{graphics}
\usepackage{yfonts}
\usepackage{here}

\usepackage{centernot}

\newtheorem{theorem}{Theorem}[section]
\newtheorem{lemma}[theorem]{Lemma}
\theoremstyle{definition}
\newtheorem{definition}[theorem]{Definition}

\theoremstyle{remark}

\numberwithin{equation}{section}
\theoremstyle{plain}

\newtheorem{corollary}[theorem]{Corollary}

\newtheorem{proposition}[theorem]{Proposition}

\def\Z{\mathbb Z}

\let\cal\mathcal

\begin{document}
\title{The Axiomatics of Free Group Rings}

\author[B.~Fine]{Benjamin Fine}	
\address{Department of Mathematics\\
Fairfield University\\Fairfield, Connecticut 06430\\
United States}	
\email{ben1902@aol.com}  

\author[A.~Gaglione]{Anthony Gaglione}	
\address{Department of Mathematics\\
United States Naval Academy\\
Annapolis, Maryland 21402\\
United States}	

\author[M. Kreuzer]{Martin Kreuzer}	
\address{Faculty of Informatik und Mathematik\\
University Passau\\94030 Passau\\
Germany}	

\author[G. Rosenberger]{Gerhard Rosenberger}	
\address{Fachbereich Mathematik\\
University of Hamburg\\
Bundestrasse 55\\
20146 Hamburg\\
Germany}	

\author[D. Spellman]{Dennis Spellman}	
\address{Department of Mathematics\\
Temple University\\
Philadelphia, Pennsylvania 19122\\
 United States}	


\begin{abstract} In \cite{FGRS1,FGRS2} the relationship between the universal and elementary theory of a group ring $R[G]$ and the corresponding universal and elementary theory of the associated group $G$ and ring $R$ was examined. Here we assume that $R$ is a commutative ring with identity $1 \ne 0$.  Of course, these are relative to an appropriate logical language $L_0,L_1,L_2$ for groups, rings and group rings respectively. Axiom systems for these were provided in \cite{FGRS1}.   In \cite{FGRS1} it was proved that if $R[G]$ is elementarily equivalent to $S[H]$ with respect to $L_{2}$, then simultaneously the group $G$ is
elementarily equivalent to the group $H$ with respect to $L_{0}$, and the ring 
$R$ is elementarily equivalent to the ring $S$ with respect to $L_{1}$.  We then let $F$ be a rank $2$ free group and $\mathbb{Z}$ be the ring of integers. Examining the universal theory of the free group ring $\Z[F]$ the hazy conjecture was made that the universal sentences true in $\Z[F]$ are precisely the universal sentences true in $F$ modified appropriately for group ring theory and the converse that the universal sentences true in $F$ are the universal sentences true in $\Z[F]$ modified appropriately for group theory. In this paper we show this conjecture to be true in terms of axiom systems for $\Z[F]$.

\end{abstract}

\maketitle

\section{Introduction} 

In \cite{FGRS1,FGRS2} the relationship was examined between the universal and elementary theory of a group ring $R[G]$ and the corresponding universal and elementary theory of the associated group $G$ and ring $R$. Here we assume that $R$ is a commutative ring with identity $1 \ne 0$. These are relative to an appropriate logical language $L_0,L_1,L_2$ for groups, rings and group rings respectively. Axiom systems for these were provided in \cite{FGRS1}.   In \cite{FGRS1} it was then proved that if $R[G]$ is elementarily equivalent to $S[H]$ with respect to $L_{2}$ then simultaneously the group $G$ is
elementarily equivalent to the group $H$ with respect to $L_{0}$ and the ring 
$R$ is elementarily equivalent to the ring $S$ with respect to $L_{1}.$ We then let $F$ be a rank $2$ free group and $\mathbb{Z}$ be the ring of integers. We call the group ring $\Z[F]$ a free group ring. It is easy to prove (see \cite{FGRS1}) that all free group rings for non-abelian free groups have the same universal theory. A {\bf Kaplansky group} $G$ is a group $G$ where the group ring $K[G]$ with $K$ a field has no zero divisors. Subsequently in \cite{FGRS3} it was shown that the class of Kaplansky groups is universally axiomatizable.  In \cite{BM} Bakulin and Myasnikov establish a set of axioms for the universal theory of the Kaplansky Groups.

Myasnikov and Remeslennikov \cite{MR} have given axiom systems for the universal theory of non-abelian free groups. In particular they proved that if $F$ is a non-abelian free group then the universal theory of $F$ is axiomatized by (see section 2 for relevant definition) the diagram of $F$, the strict universal Horn sentences of $L_0[F]$  true in $F$ and group commutative transitivity (see sections 3 and 4 for relevant definitons).   In this paper we extend this to axiom systems for free group rings and prove that the universal theory of a free group ring $\Z[F]$ is axiomatized by the diagram of $\Z[F]$, the strict universal Horn sentences of $L_2[\Z[F]]$ true in $\Z[F]$ and ring commutative transitivity when the models are restricted to group rings. Hence if $R[G]$ satisfies the diagram of $\Z[F]$ and the strict universal Horn sentences true in $\Z[F]$ and ring commutative transitivity then $R[G]$ is universally equivalent to $\Z[F]$.    Since the axiom systems are precisely of the same form modified for each theory the hazy conjecture follows.

In the next section we give the necessary preliminaries on group theory, logic and axiom systems. In section 3 we present some straightforward observations on the universal theory of free group rings. Finally in section 4 we prove the main results.

\section{Basic Preliminaries}

For a general algebraic structure, for example a group, a ring, a field or an algebra, $A$, its {\bf elementary theory} is the set of all first-order sentences in a logical language appropriate for that structure, true in $A$. Hence if $F$ is a given free group, its elementary theory consists of all first-order sentences in a language appropriate for group theory that are true in $F$. Two algebraic structures are {\bf elementary equivalent} or {\bf elementarily equivalent} if they have the same elementary theory. The Tarski theorems proved by Kharlampovich and Myasnikov and independently by Sela (see \cite{FGMRS}) say that all non-abelian free groups satisfy the same elementary theory. Kharlampovich and Myasnikov also showed that the elementary theory of free groups is decidable, that is, there is an algorithm to decide if any elementary sentence is true in all free groups or not. For a group ring they have proved that the first-order theory (in the language of ring theory) is not decidable and have studied equations over group rings especially for torsion-free hyperbolic groups.
  
The set of universal sentences in an algebraic structure $A$ that are true in $A$ is its {\bf universal theory} while two structures are {\bf universally equivalent} if they have the same universal theory. It is straightforward to show that all non-abelian free groups have the same universal theory (see \cite{FGMRS}). As part of the general solution to the Tarski theorems it was shown that a finitely generated non-abelian group is {\bf universally free} (that is has the same universal theory as a non-abelian free group) if and only if it is a limit group.  

If $F$ is a non-abelian free group then a free group ring is $\Z[F]$. In this note we consider axiom systems for  universal theories of free group rings. 

In \cite{FGRS1} we introduced three first-order languages with equality $L_{n}$, $n= 0, 1, 2$ and  listed axiom systems $T_{n}$ expressed in $L_{n}$. We view a group as a model of $
T_{0}$, a ring as a model of $T_{1}$. Moreover, we view the class of
group rings as a subclass of the model class of $T_{2}$. 

In the next section we make some observations on free group rings. Then in section 4 we review some further necessary material on elementary and universal theory. In section 4 we also present our main results.

We first review the elementary and universal theory of groups. A first-order sentence in group theory has logical symbols $\forall,\exists, \vee,\wedge, \sim$ but no quantification over sets.  A first-order theorem in a free group is a theorem that says a first-order sentence is true in all non-abelian free groups. We make this a bit more precise:

We start with a first-order language appropriate for group theory. This language, which we denote by $L_0$, is
the first-order language with equality containing a binary operation symbol $\bullet$, a unary
operation symbol $^{-1}$ and a constant symbol $1$.  A {\bf universal sentence} of $L_0$ is one
of the form $\forall{\overline{x}} \{\phi(\overline{x})\}$ where $\overline{x}$ is a tuple of
distinct variables, $\phi(\overline{x})$ is a formula of $L_0$ containing no quantifiers and
containing at most the variables of $\overline{x}$.  Similarly an {\bf existential sentence} is
one of the form $\exists{\overline{x}}\{\phi(\overline{x})\}$ where $\overline{x}$ and
$\phi(\overline{x})$ are as above. A {\bf universal-existential sentence} is one of the form
$\forall{\overline{x}} \exists{\overline{y}}\{\phi(\overline{x},\overline{y})\}$. Similarly
defined is an {\bf existential-universal sentence}. It is known that every sentence of $L_0$ is
logically equivalent to one of the form ${Q_1}_{x_1}...{Q_n}_{x_n} \phi(\overline{x})$ where
$\overline{x} = (x_1,...,x_n)$ is a tuple of distinct variables, each $Q_i$ for $i = 1,...,n$
is a quantifier, either $\forall$ or $\exists$, and $\phi(\overline {x})$ is a formula of $L_0$
containing no quantifiers and containing free at most the variables $x_1,...,x_n$.  Further
vacuous quantifications are permitted. Finally a {\bf positive sentence} is one logically
equivalent to a sentence constructed using (at most) the connectives
$\lor,\land,\forall,\exists$.   

If $G$ is a group then the {\bf universal theory} of $G$ consists of the set of all universal
sentences of $L_0$ true in $G$. We denote the universal theory of a group $G$ by $Th_\forall(G)$. 
Since any universal sentence is equivalent to the negation of an existential sentence it follows that
two groups have the same universal theory if and only if they have the same {\bf existential theory}.
The set of all sentences of $L_0$ true in $G$ is called the {\bf first-order
theory} or the {\bf elementary theory} of $G$.  We denote this by $Th(G)$. We note that being
{\bf first-order} or {\bf elementary} means that in the intended interpretation of any formula
or sentence all of the variables (free or bound) are assumed to take on as values only
individual group elements - never, for example, subsets of, nor functions on, the group in which
they are interpreted.

We say that two groups $G$ and $H$ are {\bf elementarily equivalent} (symbolically $G \equiv H$) if
they have the same first-order theory, that is $Th(G) = Th(H)$.

Prior to the solution of the Tarski problems, it was asked whether there exist finitely generated non-free {\bf elementary free groups}. By this it was meant that if all countable non-abelian free groups do have the same first-order theory do there exist finitely generated non-free groups with exactly the same first-order theory as the class of non-abelian free groups. We note that if we omit the requirement of finite generation we have $^*F = F^I/D$ is elementarily free but not free if $F$ is a free non-abelian group and $D$ is a nonprincipal ultrafilter. However $^*F$ is not finitely generated or even countable.

In the finitely generated case, the answer is yes, and both the Kharlampovich-Myasnikov solution and the Sela solution provide a complete characterization of the finitely generated elementary free groups.  In the Kharlampovich-Myasnikov formulation these are given as a special class of what are termed NTQ groups (see \cite{FGMRS}).The primary examples of non-free elementary free groups are the orientable surface groups of genus $g \ge 2$ and the non-orientable surface groups of genus $g \ge 4$ (see \cite{FGMRS}).  Recall that a {\bf surface group} is the fundamental group of a compact surface.  If the surface is orientable it is an orientable surface group otherwise a non-orientable surface group.

\section{Some Observations on Free Group Rings}

If $L$ is a first-order language with equality let us call $L$-structures 
\textbf{algebras }(as opposed to \textbf{relational systems) }provided $L$
contains no relational symbols. If $A$ and $B$ are $L$-algebras we say that 
$A$ \textbf{discriminates }$B$ provided, given finitely many pairs $%
(x_{i},y_{i})$ of unequal elements of $B$, $x_{i}\neq y_{i},~i=1,...,n,$
there is a homomorphism $\varphi \colon B\rightarrow A$ such that $\varphi
(x_{i})\neq \varphi (y_{i})$ for all $i=1,...,n$.

\smallskip

If $R$ is a ring with identity $1\neq 0$, we let $U(R)$ be its group of units.

\smallskip

\begin{definition}
Let $R$ be a commutative ring with identity $1\neq 0$ and let $G$ be a
group. An element of $U(R[G])$ is said to be a \textbf{trivial unit}
provided it has the form $ug$ where $u\in U(R)$ and $g\in G$.
\end{definition}

\smallskip

We need the following ideas:

\begin{definition}
Let $F=\langle a_{1},a_{2};~\rangle $ and let $G$ be a group. $G$ is \textbf{%
residually free }provided, for every $g\in G\backslash \{1\}$, there is a
homomorphism $\varphi \colon G\rightarrow F$ such that $\varphi (g)\neq 1$. Extending this a group is {\bf fully residually free} porvided for every finite set $g_1,...,g_n$ of nontrivial elements of $G$ there is a
homomorphism $\varphi \colon G\rightarrow F$ such that $\varphi (g_i)\neq 1$ for each $g_i$. 
\end{definition}

\begin{definition}
A group $G$ is\textbf{\ orderable }provided it admits a linear order $<$
satisfying the conditions that $hg_{1}<hg_{2}$ whenever $g_{1}<g_{2}$ and $%
g_{1}h<g_{2}h$ whenever $g_{1}<g_{2}$ as $g_{1},g_{2}$ and $h$ vary over $G$.
\end{definition}

\smallskip

It is well known that free groups are orderable. See e.g. \cite{N}. One can find,
for example in Passman's book \cite{P}, that, if $K$ is a field and $G$ is an
orderable group, then $K[G]$ has trivial units only. From that it easily
follows that if $D$ is an integral domain and $G$ is an orderable group,
then $D[G]$ has trivial units only. In analogy with groups a ring $R$ is {\bf residually}-$\Z$ if for any nontrivial $r \in R$ there exists a homomorhism $\phi$ onto $\Z$ such that $\phi(r) \ne 0$. A ring $R$ with identity $1\neq 0$ which is
discriminated by $\mathbb{Z}$ is said to be $\omega $-\textbf{residually }$%
\mathbb{Z}$. Such rings have characteristic zero so $\mathbb{Z}\leq R$. A
ring all of whose finitely generated subrings are $\omega $-residually $%
\mathbb{Z}$ is said to be \textbf{locally} $\omega $-\textbf{residually }$%
\mathbb{Z}$.

It is known that all nonabelian free groups have the same universal theory. A group is {\bf universally free} if it has the same universal theory as a nonabelian free group. Commutative transitivity ( the CT Property), that is the property that commutativity is transitive on on non-identity elements, plays a fundamental role in the study of universally free groups and hence in the solution of the Tarski problems via the following theorem due independently to Gaglione and Spellman \cite{GS} and Remeslennikov \cite{Re}.

\begin{theorem} [\cite{GS,Re}] Suppose $G$ is residually free.  Then the following are equivalent:
\begin{enumerate}[(1)]
\item $G$ is fully residually free, 

\item $G$ is commutative transitive, 

\item $G$ is universally free if non-abelian.
\end{enumerate}
\end{theorem}

In \cite{FGRS1}  an analog of this theorem within the context of free group rings was proved. 

We now make some straightforward observations about free group rings showing the ties to free groups. First is the observation  that group rings satisfy an analog of the Nielsen-Schreier theorem. In analogy with other algebraic structures a sub-(group ring) is a subset of a group ring which is also a group ring under the same operations. 

\begin{theorem} Let $A = \Z[F]$ be a free group ring. If $B$ is a sub-(group-ring) of $A$ then $B$ is a free group ring, that is $B = R[H]$ for some ring $R$ and some free group $H$.
\end{theorem}

\begin{proof} Since $F$ is free the only units in $A$ are $\pm 1$, the elements of $F$ and $-f$ for $f \in F$. Since $B$ is a subring of $A$, $H$ is included in the set of units in $A$, that is $F \cup -F$. Now if $H \subset F$, then $H$ is a free group by the Nielsen-Schreier theorem. Assume to deduce a contradiction that $H$ is not free. Then there must be a nontrivial dependence relation among the elements of $H$ say $W= 1$. If the number of elements of $W$ that comes from $-F$ is even then when these are multiplied out we get $W_1 = 1$ which is a nontrivial dependence relation in $F$ and so a contradiction. Otherwise if the number of elements of $W$ that come from $-F$ is odd then when these are multiplied out we get $-W_1 = 1 $. When this is squared we get a contradiction.  
\end{proof}

A finitely generated free group cannot be isomorphic to a proper quotient of itself \cite{MKS}. This is the well-known Hopfian property and is extremely important in free group theory. We prove two strong  generalizations of the Hopfian property to certain group rings which in turn implies that a free group ring is Hopfian within the class of group rings, that is within the class of group rings, $A[G]$ cannot be isomorphic to a proper quotient of $A$.

We first prove.

\begin{theorem} Let $G$ be a finitely generated residually finite group. Then the group ring $A = \Z[G]$  is Hopfian relative to group rings, that is $A$ cannot be isomorphic, within the class of group rings, to a proper quotient of $A$.
\end{theorem}

It is well-known in group theory that
a finitely generated residually finite group must be Hopfian.
Straightforward proofs of this can be found in the books by D. Robinson \cite{DR} and Magnus, Karrass and Solitar \cite{MKS}. Mimicking exactly the proof in Robinson's book we get the extended result for rings.

\begin{theorem} Let $R$ be a finitely generated ring. If $R$ is residually finite as a ring then $R$ cannot be isomorphic to a proper quotient of itself.
\end{theorem}

\begin{proof} To prove the result for $\Z[G]$ and $G$ a residually finite, finitely generated group we need only show that $\Z[G]$ is both finitely generated and residually finite as a ring.

Recall that a residually finite group (or ring) $G$ is fully residually finite. To see this, notice that if $g_1,..,g_n$ are finitely many nontrivial elements of $G$ then there exist finite groups $H_i$ and epimorphisms $\phi_i\colon G \rightarrow H_i$ such that $\phi_i(g_i) \ne 1$ for all $i = 1...,n$. Let $H = H_1 \times ... \times H_n$. Then $H$ is finite and there exists an epimorphism $ \phi = \phi_i \times .... \times \phi_n$ from $G \rightarrow H$ with $\phi(g_i) \ne 1$ for $i = 1,..,n$.

Now let $G$ be a finitely generated residually finite group and $A = \Z[G]$ the integral group ring over $G$. Let $S = \{s_1,..,s_n\}$ be a set of generators for $G$ and $S^{-1} = \{s_1^{-1},...,s_n^{-1}\}$. Then $\{S \cup S^{-1}\}$ gives a finite set of ring generators for $A$.  Hence $A$ is finitely generated as a ring.  We must show that $A$ is residually finte. 

Suppose that $r = n_1g_1 + .... + n_mg_m \ne 0$ is an arbitrary nonzero element in $A$ with distinct elements $g_1,...,g_m \in G$. Since $G$ is fully residually finite there exists a finite group $H$ and a group epimorphism $\phi\colon  G \rightarrow H$ which does not annihilate any $g_ig_j^{-1}$ with $1 \le i,j \le m$.  We then get an epimorphism $\psi\colon \Z[G] \rightarrow \Z[H]$ by $\psi\colon \sum_j n_jg_j \rightarrow \sum_j n_j \phi(g_j)$ and further $\psi(r) = \psi(n_1g_1 + ... + n_mg_m) = n_1\phi(g_1) + ... + n_m\phi(g_m)  \ne 0$ in $\Z[H]$. 

Choose a prime $p$ large enough so that $p \nmid n_i$ for $1 \le i \le m$. Then there exists a ring epimorphism $\pi\colon  \Z[H] \rightarrow \Z_p[H]$ given by
$$ \pi\colon  \sum_{h_t} n_th_t \rightarrow \sum_{h_t} \overline{n_t}h_t$$
where $\overline{n_t}$ is the residue class for $n_t$ mod $p$. Then 
$$\pi(r)  = \pi(n_1g_1 + ... + n_mg_m) = \overline{n_1}\phi(g_1) _ .... +\overline{n_m}\phi(g_m) \ne 0 $$
in $\Z_p[H].$
The ring $\Z_p[H]$ is a finite ring so $\pi\psi$ is an epimorphism from $\Z[G]$ to $\Z_p[H]$ which does not annihilate $r$. Since $r$ is an arbitrary nonzero element of $A$ it follows that $A = \Z[G]$ is residually finite.

Therefore $\Z[G]$ is both finitely generated and residually finite as a ring and hence is Hopfian.
\end{proof}

Since finitely generated free groups are residually finite we  get directly.

\begin{corollary} Let $F$ be a finitely generated free group. Then the free group ring $\Z[F]$ is Hopfian.
\end{corollary}

Essentially the same proof gives the following stronger generalization.

\begin{theorem} Let $R$ be a finitely generated residually finite commutative ring with identity $1 \ne 0$ and $G$ a finitely generated residually finite group. Then the group ring $R[G]$ is finitely generated and residually finite and hence Hopfian.
\end{theorem}

The final straightforward result is the Howson property for free group rings. Recall that for free groups the Howson property says that the intersection of two finitely generated subgroups of a free group is again finitely generated \cite{LS}. The same result follows directly for finitely generated sub-(group rings) of a free group ring.

\begin{theorem} Let $A$, $B$ be two finitely generated  sub-(group rings) of a finitely generated free group ring $\Z[F]$. Then $A \cap B$ is finitely generated.
\end{theorem}

\begin{proof} Since $A$ and $B$ are finitely generated sub-(group rings) then $A = \Z[G]$ and $B = \Z[H]$ where $G$ and $H$ are finitely generated subgroups of the free group $F$. The result follows from the Howson property in $F$.
\end{proof} 

The proof generalizes  to show that if a group $G$ satisfies the Howson property as a group then the integral group ring $\Z[G]$    satisfies the Howson property as a group ring.

\section{Axiomatics for the Universal Theory of Free Group Rings}

A group $G$ is an $H$-group if it contains as a subgroup an isomorphic copy of $H$. Further if $L_i$ is one of our three main languages, (groups, rings and group rings), and $G$ is an appropriate model of $L_i$ then $L_i[G]$-structures are those models which include the elements of $G$ as constants. For the next we refer to Chang and Keisler \cite{CK}. We need to define a quasivariety.

Recall that $L_i$ for $i = 0,1,2$ are appropriate logical languages for groups ($L_0$),rings ($L_1$) and group rings ($L_2$).  A \textbf{quasi-identity }of $L_{i}$ is a universal sentence of the form 
$$
\forall \overline{x}\left( \bigwedge\limits_{i}(S_{i}(\overline{x})=s_{i}(%
\overline{x}))\rightarrow \left( T(\overline{x})=t(\overline{x})\right)
\right)
$$
where $\overline{x}$ is a tuple of distinct variables and $S_{i}(\overline{x}%
),s_{i}(\overline{x}),$ $T(\overline{x})$ and $t(\overline{x})$ are terms of 
$L_{i}$ involving at most the variables in $\overline{x}$.

A \textbf{%
quasivariety }of groups is the model class of a set of quasi-identities of $%
L_{0}$ together with the group axioms.

A formula $\phi$ of a language $L$ is a {\bf basic Horn formula} if and only if $\phi$ is a disjunction of formulas $\phi_i$, $\phi = \phi_1 \vee \cdots \vee \phi_n$ where at most one of the formulas $\phi_i$ is an atomic formula, the rest being the negation of atomic formulas. Then a basic Horn formula is a {\bf strict basic Horn formula} if and only if some $\phi_i$ is an atomic formula.  A {\bf Horn formula} is a conjunction of basic Horn formulas. Then a universal sentence $\forall \overline{x}\phi(\overline{x})$ is a {\bf strict universal Horn sentence} provided $\phi(\overline{x})$ is a strict basic Horn formula. In the case that $L$ is an algebra, that is has no relational symbols, then a strict universal Horn sentence in $L$ is up to logical equivalence a quasi-identity of $L$. For a model $A$ of a language its {\bf diagram} is the set of atomic and negated atomic sentences of $L[A]$ true in $A$. Hence the diagram of a group $G$ (a group ring $R[G]$) is the set of atomic and negated atomic sentences of $L_0[G]$ ($L_2(R[G]$) true in $G$ (in $R[G]$)

Finally we call a group ring which is universally equivalent with respect to the language $L_2$ to the integral group ring $\Z[F]$ where $F$ is a non-abelian free group a {\bf universally free group ring}.

Myasnikov and Remeslennikov \cite{MR} proved the following. A group is CSA or conjugately separated abelian if maximal abelian subgroups are malnormal. A subgroup $M$ of a group $G$ is malnormal if $g^{-1}Mg \cap M \ne \{1\}$ implies that $g \in M$. The CSA property implies CT.

\begin{theorem} Let $G$ be a non-abelian CSA group which is equationally Noetherian. Then the universal theory of $G$ with respect to $L_0[G]$ is axiomatizbale by the set $Q$ of quasi-identities of $L_0[G]$ true in $G$ together with CT when the models are restricted to be $G$-groups.
\end{theorem}

It was subsequently shown by Fine, Gaglione and Spellman \cite{FGS} that the equationally Noetherian condition is not necessary and hence we have the following.

\begin{theorem}Let $F$ be a non-abelian free group. Then any $F$-group $H$ which is a model of the set $Q$ of quasi-identities of $L_0[F]$ true in $F$ together with CT is already a model of the universal theory of $F$ with respect to $L_0[F]$.
\end{theorem}

Our result is a direct analog of this result in the context of group rings containing the free group ring $\Z[F]$. Recall that if $R[G]$ is a group ring and $x \in R[G]$ then $\Gamma(x)$ means that $x \in G$ while ${\cal{P}}(x)$ indicates the $x \in R$. We also make the convention that $G = \Gamma(R[G])$ and $R = {\cal{P}}(R[G])$.  In the language $L_2[\Z[F]]$ commutative transitivity is expressed as
$$RCT\colon  \forall x,y,z((\sim {\cal{P}}(y) \wedge (xy=yx) \wedge (yz = zy)) \rightarrow (xz=zx)).$$

It is known that that the elementary theory of a CT group $G$ is axiomatized by the set $H(G)$ of Horn sentences true in $G$ together with CT (see \cite{MR}).  Further Myasnikov and Remeslennikov have proved \cite{MR} the universal theory of a CSA group $G$ is given by the diagram of $G$, the strict universal Horn sentences of $L_0[G]$ true in $G$ and commutative transitivity. Myasnikov and Remeslennikov required $G$ to be equationally Noetherian but it was shown in \cite{FGS} that equationally Noetherian is superfluous. In light of this result and the examples of universal sentences in free group rings the hazy conjecture was made that the universal theory of a free group rings consisted of the universal theory of free groups appropriately modified to group ring theory and vice versa. Our main result is the following which shows this to be true in terms of axiom systems. We obtain a result similar to the theorem on elementary theory. 

\begin{theorem}\label{thm43} Let $G$ be a group and suppose that the group ring $R[G]$ satisfies the diagram of the free group ring $\Z[F]$, the strict universal Horn sentences of $L_2[\Z[F]]$ true in $\Z[F]$ and ring commutative transitivity. Then $R[G] \equiv_\forall \Z[F]$ with respect to $L_2[\Z[F]]$.
\end{theorem}

Myasnikov and Remslennikov prove more. If $G$ is any finitely generated non-abelian CSA group then the universal theory of $G$ with respect to $L_0[G]$ is axiomatizable by the diagram of $G$, the set of quasi-identities of $L_0[G]$ true in $G$ and CT. The same proof as for the free group ring $\Z[F]$ goes through to give:

\begin{theorem}\label{thm44} Let $H$ be a nonabelian torsion-free CSA group and suppose that $G$ is a torsion-free group. If the group ring $R[G]$ satisfies the diagram of the group ring $\Z[H]$ the strict universal Horn sentences of $L_2[\Z[H]]$ true in $\Z[H]$ and ring commutative transitivity then $R[G] \equiv_\forall \Z[H]$ with respect to $L_2[\Z[H]]$.
\end{theorem}

The proofs of the results follow as corollaries of the proof of the next theorem which is the group ring analog of the Gaglione-Spellman-Remeslennikov result for groups.

\begin{theorem}\label{thm45} Let $F$ be a non-abelian free group. Let $R[H]$ be a group ring containing $\Z[F]$ which is a model of the set $S$ of strict universal Horn sentences of $L_2[\Z[F]]$ true in $\Z[F]$. Then the following are equivalent: 
\begin{enumerate}[(1)]
\item $R[H]$ is CT
\item $R[H]$ is a model of the universal theory of $\Z[F]$ with respect to $L_2[\Z[F]]$
\end{enumerate}
\end{theorem}

\begin{proof} To prove this we need some preliminary results. First note that (2) implies (1) is trivial since $\Z[F]$ is commutative transitive by \cite{M}.  The difficulty is in the proof that (1) implies (2). To do this we first need the following proposition and it corollary.  

\begin{proposition} Let $G$ be a group and $g \in G$ with infinite order. Then $1-g$ is not a zero divisor in the integral group ring $\Z[G]$.
\end{proposition}

\begin{proof} (of proposition)  For the proof of the proposition we need the following preliminary lemma.

\begin{lemma} Let $\phi\colon \Z[G] \rightarrow \Z[G]$ be the endomorphism of the additive group of $\Z[G]$ given by $\phi(x) = gx$ for all $x \in \Z[G]$. Then $0$ is the only fixed point of $\phi$. Similarly if $\psi\colon \Z[G] \rightarrow \Z[G]$ is given by $\psi(y) = yg$ for all $y \in \Z[G]$ then $0$ is the only fixed point of $\psi$.
\end{lemma}

To prove the lemma we will give the arguments for $\phi$. The argument for $\psi$ is analogous. Suppose to deduce a contradiction that $r_1g_1 + ... + r_ng_n$ is a non-zero fixed point of $\phi$. Here $i < j$ then $g_i \ne g_j$ and $r_i \in \Z \setminus \{0\}$ for $1 \le i \le n$. Let $S = \{g_1,...,g_n\}$. Then $\phi$ must restrict to a permutation $\pi$ of $S$ with $\pi$ having finite order $N$. Then $\phi^N(g_i) = g_i$ for $i = 1,...,n$ so $g^Ng_i = g_i$ for $i = 1,...,n$.  From $g^Ng_1 = g_1$ we get that $g^N = 1$ contradicting the hypothesis that $g$ has infinite order. This shows that $0$ can be only fixed point of $\phi$. An analogous argument follows for $\psi$ completing the proof of the lemma.

We now return to the proof of the proposition, It follows that $ker(1-\phi) = \{0\}= ker(1-\psi)$. Then $0 = (1-g)x = (1 -\phi) (x)$ implies that $x = 0$, Similarly $0 = y(1-g) = 0$ implies that $y = 0$, hence $1-g$ is not a zero divisor in $\Z[G]$ and the proposition is proved.
\end{proof} 

If $g$ has finite order $n$ then $1-g$ is a zero divisor in $\Z[G]$. A straightforward induction establishes the following corollary,

\begin{corollary} Let $G$ be a torsion-free group. Then $(1-g_1)(1-g_2)  \cdots (1-g_n) = 0$  in $\Z[G]$ if and only if at least one $g_i = 1$.
\end{corollary}

Now we return to the proof of Theorem \ref{thm45} and the difficult part that (1) implies (2). Let $G$ be a torsion-free group and let $H$ be a torsion-free $G$-group. Recall that a {\bf primitive sentence} is an existential sentence whose matrix is a conjunction of atomic sentences and negations of atomic sentences. 

$H$ is a model in $L_0[G]$ of $Th_{\forall}(G)$ if and only if every primitive sentence of $L_0[G]$ true in $H$ is also true in $G$. Equivalently $H$ is a model of $Th_{\forall}(G)$ if and only if every negated primitive sentence of $L_0[G]$ true in $G$ is also true in $H$. Now consider a primitive sentence
$$\exists \overline{x} (\wedge_{i=1}^p (u_i(\overline{x}) = 1 ) \wedge \wedge_{j=1}^q (w_j(\overline{x}) \ne 1))$$ of $L_0[G]$ where $\overline{x} = (x_1,...,x_k)$ is a tuple of distinct variables. The negation of this is equivalent to
$$\forall \overline{x} (\vee_{i=1}^p (u_i(\overline{x}) \ne 1 ) \vee \vee_{j=1}^q (w_j(\overline{x}) = 1))$$
which is equivalent to
$$\forall \overline{x} (\sim (\wedge_{i=1}^p (u_i(\overline{x}) = 1 )) \vee \vee_{j=1}^q (w_j(\overline{x}) =1 ))$$
which in turn is equivalent to
$$\forall \overline{x} ((\wedge_{i=1}^p u_i(\overline{x}) = 1 ) \rightarrow \vee_{j=1}^q (w_j(\overline{x}) = 1))$$

In view of the fact that $(1-g_1) \cdots (1-g_q) = 0$ in $\Z[G]$ if and only if at least one $g_j = 1$ the above sentence of $L_0[G]$ is  equivalent to the following strict universal Horn sentence of $L_2[\Z[G]]$
$$ \forall \overline{x} ((\wedge_{\nu= 1}^k \Gamma(x_\nu) \wedge \wedge_{i=1}^p (u_i(\overline{x}) =1 )) \rightarrow ((1-w_1(\overline{x} ))\cdots (1- w_q(\overline{x})) = 0))$$
It follows that if $Z[H]$ is a model of the above strict universal Horn sentences of $L_2[\Z[G]]$ true in $\Z[H]$ then $H$ is a model of $Th_\forall(G)$.

Now let $F$ be a nonabelian free group, let $G$ be an $F$-group and let $R[G]$ be the group ring where $R$ is a commutative ring with $1 \ne 0$ and of characteristic $0$. We may view $R$ as an extension of $\Z$. Suppose that $R[G]$ is commutative transitive and is a model of the strict universal Horn sentences of $L_2[\Z[F]]$ true in $\Z[F]$. Now let $R_0$ be a finitely generated unital subring of $R$. We claim that $R_0$ is residually-$\Z$. 

To see this suppose that $r_1,...,r_n$ generate $R_0$. Then $R_0$ is a homomorphic image of the integral polynomial ring $\Z[x_1,...,x_n]$ where we get a surjective homomorphism onto $R_0$ by $x_i \mapsto r_i$ with $i=1,...,n$. By the Hilbert Basis Theorem $\Z[x_1,...,x_n]$ is Noetherian. Hence if $K = ker(\Z[x_1,...,x_n] \rightarrow R_0)$ then $K$ is finitely generated as an ideal. Suppose that $g_1(x_1,...,x_n),...,g_k(x_1,...,x_n)$ generate $K$ as an ideal. Suppose to deduce a contradiction that $s \in R_0\setminus \{0\}$ is annihilated  by every retraction $\rho\colon R_0 \rightarrow \Z$.

Letting $s = f(r_1,..,r_n)$ we have the following quasi-identity
$$ \forall x_1,....,x_n (\wedge_{i=1}^k \wedge (g_i(x_1,...,x_n) = 0) \rightarrow (f(x_1,...,x_n) = 0))$$
true in $\Z$. Thus the following strict universal Horn sentence of $L_2[\Z[G]]$ is true in $\Z[G]$;
$$ \forall x_1,...,x_n (( \wedge_{j=1}^n {\cal{P}}(x_j) \wedge \wedge_{i=1}^k( g_i(x_1,...,x_n) = 0)) \rightarrow (f(x_1,...,x_n)= 0))$$
hence it is true in $R[G]$. However that is contradicted by the substitution
$$x_i \mapsto r_i, i = 1,...,n$$
since $f(r_1,...,r_n) \ne 0$.

This contradiction shows that $R_0$ is residually-$\Z$. 

We now claim that $R$ and therefore $R_0$ also is an integral domain.  Suppose not. Let $(u,v) \in (R\setminus\{0\})^2$ with $uv = 0$. Fix $a,b \in F \subset G$ with $ab \ne ba$. Now in $R[G]$, $ua$ commutes with $vb$ which commute with $ub$. By commutative transitivity $ua$ and $ub$ commute. Therefore $u^2ab = u^2ba$ in $R[G]$. The strict universal Horn sentence of $L_2[\Z[G]]$
$$ \forall x (( {\cal{P}}(x) \wedge (x^2 = 0)) \rightarrow (x = 0))$$
is true in $\Z[F]$. Hence it holds in $R[G]$ and therefore $u^2 \ne 0$. By uniqueness of representation $ab = ba$ giving a contradiction. Therefore both $R$ and $R_0$  are integral domains. 

It follows then that $R_0$ is $\omega$-residually-$\Z$ since if $s_1,..,,s_n$ are finitely many non-zero elements of $R_0$ there is a retraction $\rho\colon  R_0 \rightarrow \Z$ which does not annihilate their product. From $\rho(s_1) \cdots \rho(s_n) = \rho(s_1 \cdots s_n) \ne 0$ we get $\rho(s_i) \ne 0$ for $i = 1,...,n$. Since $R_0$ was an arbitrary finitely generated unital subring of $R$ it follows that $R$ is locally $\omega$-residually-$\Z$. It follows that $R$ is a model of $Th_\forall(\Z)$. Hence $R$ embeds in an ultrapower $\Z^I/D$ and so $R[F]$ is a model of $Th_\forall(\Z[F])$. Since $\Z[F]$ embeds in $R[F]$ every existential sentence true in $\Z[F]$ is also true in $R[F]$.  Hence $\Z[F] \equiv_{\forall} R[F].$
Since $F$ is torsion-free, for each integer $n$ the following quasi-identity
$$\forall x (( x^n =1) \rightarrow (x=1))$$
holds in $F$. It follows that for each integer $n \ge 2$ the strict universal Horn sentence
$$\forall x ((\Gamma(x) \wedge (x^n = 1)) \rightarrow (x = 1))$$
holds in $\Z[F].$ Hence it holds in $R[G]$ and $G$ is torsion-free. We claim that $G$ is a model of $Th_\forall(F)$. Let $\Theta$ be a negated primitive sentence true in $F$. Then $\Theta$ is equivalent to  a strict universal Horn sentence
$$\forall \overline{x} ((\wedge_{\nu =1 }^k \Gamma(x_{\nu}) \wedge \wedge_{i=1}^p (u_i(\overline{x}) = 1)) \rightarrow ((1-w_1(\overline{x})) \cdots (1- w_q(\overline{x})) = 0))$$
true in $\Z[F]$. Hence it is true in $R[G]$ and since $\Z[G]$ is contained in $R[G]$ it is true in $Z[G]$ It follows that $\Theta$ is true in $G$ and thus as claimed $G$ is a model of $Th_\forall(F)$.

It follows that $G$ embeds in an ultrapower $F^J/E$ of $F$. Then $R[G]$ embeds into $(R[F])^J/E$ so $R[G]$ is a model of $Th_{\forall}(R[F])$. Since $R[F]$ embeds into $R[G]$ every existential sentence true in $R[F]$ is also true in $R[G]$. It follows that $R[F] \equiv_\forall R[G]$. From $\Z[F] \equiv_\forall R[F]$ and $R[F] \equiv_\forall R[G]$ we get by transitivity of $\equiv_\forall$ that $R[G] \equiv_\forall \Z[F]$. This shows that in Theorem 4.3 (1) implies (2) completing the proof.
 \end{proof}

Theorem 4.3 follows from the statement of Theorem \ref{thm45}. To see this,  note that since $R[G]$ satisfies the diagram of $\Z[F]$, then $R[G]$ contains $\Z[F]$;
so $R[G]$ is a model of $Th_{\forall}(Z[F])$ by Theorem \ref{thm45}. So every universal sentence of $L_2[\Z[F]]$ true in $Th_{\forall} \Z[F]$ is true in $R[G]$. But since $\Z[F] \subset
 R[G]$, every universal sentence of $L_2[\Z[F]]$ true in $R[G]$ must also be true in $\Z[F]$. Therefore, $R[G] \equiv_\forall R[F]$ and  Theorem \ref{thm43} follows.
Finally, the proof of Theorem \ref{thm44} goes through verbatim from the proof of Theorem \ref{thm45}.


\end{document}